\documentclass{amsart}
\usepackage{amsmath, amsthm, amssymb}
\usepackage[mathscr]{eucal}
\usepackage{enumerate}
\usepackage[all]{xy}
\usepackage{enumitem}
\usepackage{tikz}
\usetikzlibrary{arrows, chains, matrix, positioning, scopes}

\newcommand{\script}[1]{{\mathcal{#1}}}
\newcommand{\A}{\script{A}}
\newcommand{\B}{\script{B}}

\newcommand{\I}{\script{I}}

\newcommand{\Y}{\Zme^I}

\newcommand{\frah}{\mathfrak{H}}

\newcommand{\Zme}{\mathcal{Z}}
\newcommand{\E}{\mathcal{E}}

\newcommand{\Hil}{\mathcal{H}}
\newcommand{\ip}[2]{\left( \left. #1 \, \right| \, #2 \right)}
\newcommand{\hip}[2]{\langle \! \langle #1, #2 \rangle \! \rangle}

\newcommand{\norm}[1]{{\left\| #1 \right\|}}

\newcommand{\supp}{\operatorname{supp}}

\newcommand{\id}{{\rm{id}}}
\newcommand{\spa}{\operatorname{span}}

\newcommand{\image}{\operatorname{im}}

\newcommand{\go}{{G^{\scriptscriptstyle{(0)}}}}
\newcommand{\gtwo}{G^{\scriptscriptstyle{(2)}}}
\newcommand{\ho}{{H^{\scriptscriptstyle{(0)}}}}

\newcommand{\lo}{L^{\scriptscriptstyle{(0)}}}
\newcommand{\ltwo}{L^{\scriptscriptstyle{(2)}}}
\DeclareMathOperator{\Ind}{Ind}
\newcommand{\indhpi}{{\rm{Ind}}_{\ho}^H \pi}
\newcommand{\indhpitilde}{{\rm{Ind}}_{\ho}^H \tilde{\pi}}
\newcommand{\Zop}{Z^{\text{op}}}

\newcommand{\gcra}{\Gamma_c(G, r^*\A)}

\newcommand{\acrossh}{\A \rtimes_{\alpha, r} H}
\newcommand{\rhoacrossgh}{\rho^*\A \rtimes_{\sigma, r} X^H}
\newcommand{\icrossh}{\I \rtimes_{\alpha\vert_I, r} H}
\newcommand{\rhoicrossgh}{\rho^*\I \rtimes_{\sigma\vert_{\rho^*I}, r} X^H}
\newcommand{\amodicrossh}{\A/\I \rtimes_{\alpha^I, r} H}
\newcommand{\rhoamodicrossgh}{\rho^*(\A/\I) \rtimes_{\sigma^{\rho^*I}, r} X^H}

\newcommand{\fullicrossh}{\I \rtimes_{\alpha\vert_I} H}
\newcommand{\fullrhoicrossgh}{\rho^*\I \rtimes_{\sigma\vert_{\rho^*I}} X^H}

\newcommand{\EI}{J^\rho}
\newcommand{\BI}{J}
\newcommand{\J}{K}
\newcommand{\Jtwiddle}{K^\rho}

\newtheorem{prop}{Proposition}[section]
\newtheorem{thm}[prop]{Theorem}
\newtheorem{cor}[prop]{Corollary}
\newtheorem{lem}[prop]{Lemma}
\theoremstyle{definition}
\newtheorem{defn}[prop]{Definition}

\newtheorem{exmp}[prop]{Example}
\newtheorem{rem}[prop]{Remark}

\newlist{thmenum}{enumerate}{10}
\setlist[thmenum,1]{label=\textnormal{(\alph*)}}
\setlist[thmenum,2]{label=\textnormal{(\roman*)}}

\newlist{altenum}{enumerate}{10}
\setlist[altenum,1]{label=\textnormal{(\roman*)}}
\setlist[altenum,2]{label=\textnormal{(\alph*)}}

\allowdisplaybreaks

\title{Equivalence and Exact Groupoids}
\author{Scott M. LaLonde}
\address{Department of Mathematics, The University of Texas at Tyler, Tyler, TX, 75799}
\email{slalonde@uttyler.edu}

\begin{document}

\begin{abstract}
	Given two locally compact Hausdorff groupoids $G$ and $H$ and a $(G,H)$-equivalence $Z$, one can construct the associated linking 
	groupoid $L$. This is reminiscent of the linking algebra for Morita equivalent $C^*$-algebras. Indeed, Sims and Williams reestablished 
	Renault's equivalence theorem by realizing $C^*(L)$ as the linking algebra for $C^*(G)$ and $C^*(H)$. Since the 
	proof that Morita equivalence preserves exactness for $C^*$-algebras depends on the linking algebra, the linking groupoid should serve the 
	same purpose for groupoid exactness and equivalence. We exhibit such a proof here.
\end{abstract}

\maketitle

\section{Introduction}
\label{sec:intro}
The notion of equivalence for locally compact groupoids is a powerful tool that has many interesting implications for groupoid $C^*$-algebras.
The definition, originally developed by Renault, seems to have first appeared in print in \cite[Def. 2.1]{MRW}. In short, two groupoids $G$ 
and $H$ are \emph{equivalent} if there is a locally compact Hausdorff space $Z$ that admits suitable commuting left and right actions of $G$ and $H$, 
respectively. At a glance, this definition should remind one of Morita equivalence for $C^*$-algebras. Indeed, Renault proved 
in \cite[Cor. 5.4]{renault87} that the full groupoid $C^*$-algebras $C^*(G)$ and $C^*(H)$ are Morita equivalent via a completion of $C_c(Z)$. This 
result subsumes many classical results, including Green's symmetric imprimitivity theorem \cite[Cor. 4.11]{TFB2}. It has since been extended to groupoid crossed 
products \cite{mw08} and Fell bundle $C^*$-algebras \cite{muhly-williams}.

Since many $C^*$-algebraic properties (such as nuclearity and exactness) are preserved under Morita equivalence, it seems plausible that certain 
desirable properties of groupoids should be invariant under Renault equivalence. For example, it is already known 
\cite[Thms. 2.2.17 \& 3.2.16]{ananth-renault} that equivalence preserves amenability. This paper deals specifically with the property of exactness for 
groupoids, which directly generalizes Kirchberg and Wassermann's notion of exactness for groups \cite{kw99}. We show that this property is preserved
under equivalence by mimicking a purely $C^*$-algebraic argument. In \cite[Prop. A.10]{katsura}, Katsura showed that if $A$ is a $C^*$-algebra and 
$A_0 \subseteq A$ is a full, hereditary subalgebra, then $A_0$ is exact if and only if $A$ is. By considering the linking algebra, one can then show that 
Morita equivalence preserves exactness for $C^*$-algebras. An argument like this one works at the level of groupoids, but it requires a suitable analogue 
of the linking algebra.

If $G$ and $H$ are two locally compact Hausdorff groupoids and $Z$ is a $(G,H)$-equivalence, one can construct an object called the 
\emph{linking groupoid} of $G$ and $H$.  Its underlying set is the topological disjoint union $L = G \sqcup Z \sqcup \Zop \sqcup H$, and the 
groupoid operations restrict to the usual ones on $G$ and $H$. Hence it contains both $G$ and $H$ as closed subgroupoids. This groupoid is 
described in Muhly's notes \cite[Rmk. 5.35]{muhly}, though the author notes the unfortunate absence of a Haar system on $L$. In \cite{sims-williams2012}, Sims 
and Williams were able to equip $L$ with a Haar system, and they recovered Renault's equivalence theorem by realizing $C^*(L)$ as the 
appropriate linking algebra. (The Haar system on $L$ was also constructed independently by Paravicini in \cite{paravicini}.) Additionally,
Sims and Williams showed that $C_r^*(G)$ and $C_r^*(H)$ are Morita equivalent via the linking algebra $C_r^*(L)$. In a later paper 
\cite{sims-williams2013}, they proved similar equivalence theorems for full and reduced Fell bundle $C^*$-algebras.

Given its connection to the linking algebra of two groupoid $C^*$-algebras, the linking groupoid should provide the key to 
reproducing Katsura's result in the realm of groupoids. This paper is devoted to a proof in this vein, which appears in Section 4. We begin by laying out 
some preliminaries on equivalence and the linking groupoid in Section 2, and Section 3 is devoted to the proof a crucial technical result. We explore 
a short application of the main result in Section 5.

\section{Groupoids, Equivalence, and Exactness}
Let $G$ denote a locally compact Hausdorff groupoid. We write $\go$ for the unit space of $G$ and $\gtwo$ for the set of composable pairs, and
$r, s : G \to \go$ denote the range and source maps, respectively. We assume that all groupoids are second countable and carry 
continuous Haar systems unless otherwise specified. 

Recall that a locally compact Hausdorff space $Z$ is a \emph{(left) $G$-space} if there is a continuous open surjection $r_Z : Z \to \go$ and 
a continuous map $(\gamma, z) \mapsto \gamma \cdot z$ from $G {_s *}_{r_Z} Z$ to $Z$ satisfying:
\begin{enumerate}
	\item if $(\gamma, \eta) \in \gtwo$ and $(\eta, z) \in G * Z$, then $(\gamma \eta) \cdot z = \gamma \cdot (\eta \cdot z)$,
	\item $r_Z(z) \cdot z = z$ for all $z \in Z$.
\end{enumerate}
We say that the $G$-action is \emph{free} if $\gamma \cdot z = z$ implies $\gamma \in \go$, and it is \emph{proper} if the map $G * Z \to Z \times Z$
defined by $(\gamma, z) \mapsto (\gamma \cdot z, z)$ is proper. If $G$ acts both freely and properly on $Z$, we call $Z$ a \emph{principal}
$G$-space.

\begin{rem}
	The definition of a right $G$-action is analogous to the one given above, though the map $Z \to \go$ is usually denoted by $s_Z$. 
	Also, when there is no chance of confusion we will suppress the 
	subscripts on the structure maps and simply write $r$ and $s$ in place of $r_Z$ and $s_Z$.
\end{rem}

\begin{defn}
\label{defn:gpoidequiv}
	Let $G$ and $H$ be second countable locally compact Hausdorff groupoids. A second countable locally compact Hausdorff space $Z$ 
	is called a \emph{$(G, H)$-equivalence}\index{groupoid equivalence} it if satisfies the following conditions:
	\begin{thmenum}
		\item $Z$ is a principal left $G$-space;
		\item $Z$ is a principal right $H$-space;
		\item the actions of $G$ and $H$ commute;
		\item the range map $r_Z : Z \to \go$ induces a homeomorphism $Z/H \cong \go$;
		\item the source map $s_Z : Z \to \ho$ induces a homeomorphism $G \backslash Z \cong \ho$.
	\end{thmenum}
\end{defn}

There are two examples that are particularly noteworthy. If $G$ is a locally compact Hausdorff groupoid, then $G$ is a $(G, G)$-equivalence. Also, a 
transitive groupoid is equivalent to any of its isotropy groups \cite[Ex. 2.2]{MRW}. Beyond these very special cases, it is not obvious that there are any interesting 
examples of groupoid equivalences. However, the following example shows that they are quite prevalent.

\begin{exmp}
	Let $H$ be a locally compact Hausdorff groupoid, and let $Z$ be a principal right $H$-space. Then $H$ acts diagonally on $Z *_s Z$,
	and we define $Z^H = (Z *_s Z)/H$. Two equivalence classes $[z, w]_H, [x, y]_H \in Z^H$ are composable when $[w] = [x]$ in $Z/H$,
	and we define
	\[
		[z, w]_H [w, y]_H = [z, y]_H \quad \text{ and } \quad [z, w]_H^{-1} = [w, z]_H.
	\]
	It can be checked that these operations make $Z^H$ into a locally compact Hausdorff groupoid, which is second countable whenever $Z$
	is \cite[Prop. 1.92]{geoff}. Furthermore, the range and source maps are given by
	\[
		r([z, w]_H) = [z, z]_H \quad \text{ and } \quad s([z, w]_H) = [w, w]_H,
	\]
	which allows us to identify the unit space of $Z^H$ with $Z/H$. The groupoid $Z^H$ is called the \emph{imprimitivity groupoid} associated
	to the $H$-space $Z$. It admits a natural action on $Z$, which makes $Z$ into a $(Z^H, H)$-equivalence. Thus any principal groupoid space
	gives rise to a groupoid equivalence in a canonical way. Indeed, this construction is prototypical: if $Z$ is a $(G,H)$-equivalence, then $Z^H$ and 
	$G$ are naturally isomorphic via the map
	\begin{equation}
	\label{eq:gzw}
		[z, w] \mapsto {_G[}z, w],
	\end{equation}
	where ${_G[}z, w] \in G$ is the unique element satisfying ${_G[}z, w] \cdot w = z$.
\end{exmp}

It is well-known that groupoid equivalence induces Morita equivalence of the associated groupoid $C^*$-algebras. Indeed, \cite[Cor. 5.4]{renault87}
guarantees that $C_c(Z)$ completes to a $C^*(G) - C^*(H)$-imprimitivity bimodule. Sims and Williams developed an alternative proof of Renault's
result using the linking groupoid \cite[Cor. 5.2]{sims-williams2012}. One advantage to their approach is that the proof descends nicely to the level of 
reduced groupoid $C^*$-algebras \cite[Thm. 4.1]{sims-williams2012}. Since it is a crucial tool in this paper, we outline the construction of the linking 
groupoid here.

Given a $(G, H)$-equivalence $Z$, define the \emph{opposite space} $\Zop = \{ \overline{z} : z \in Z \}$ to be a homeomorphic copy of $Z$, but
let $H$ and $G$ act on the left and right, respectively:
\[
	r(\overline{z}) = s(z), \quad s(\overline{z}) = r(z), \quad \eta \cdot \overline{z} = \overline{z \cdot \eta^{-1}}, \quad \overline{z} \cdot \gamma = 
		\overline{\gamma^{-1} \cdot z}
\]
for $\eta \in H$, $\gamma \in G$. It is straightforward to check that this makes $\Zop$ into an $(H, G)$-equivalence. We then define the 
\emph{linking groupoid} to be the topological disjoint union
\[
	L = G \sqcup Z \sqcup \Zop \sqcup H
\]
with unit space
\[
	\lo = \go \sqcup \ho.
\]
The set of composable pairs is
\[
	\ltwo = \bigl\{ (x, y) \in L \times L : s(x) = r(y) \bigr\},
\]
and one simply needs to specify the multiplication and inversion operations on $L$. These operations restrict to the usual ones on $G$ and $H$, and
the left $G$-action on $Z$ lets us define
\[
	\gamma z = \gamma \cdot z
\]
for $\gamma \in G$ and $z \in Z$. Similar products are defined using the actions of $G$ and $H$ on $Z$ and $\Zop$. If $(z, \overline{w}) \in \ltwo$ with 
$z \in Z$ and $\overline{w} \in \Zop$, then 
we put
\[
	z \overline{w} = {_G [} z, w] \in G,
\]
where ${_G [} z, w]$ is defined as in \eqref{eq:gzw}. Similarly, we can define $\overline{z} w = [z, w]_H$. Finally, inversion is given on $Z$ (and hence on $\Zop$) by 
$z^{-1} = \overline{z}$. It is shown in \cite[Lem. 2.1]{sims-williams2012} that these operations make $L$ into a locally compact
Hausdorff groupoid, called the \emph{linking groupoid} of $G$ and $H$. 


We have already mentioned the groupoid $C^*$-algebra $C^*(L)$, which only exists if $L$ is equipped with a Haar system. 
However, it is not obvious that if $G$ and $H$ have Haar systems $\{\lambda^u\}$ and $\{\beta^v\}$, respectively, then $L$ possesses a Haar system. 
Fortunately, this was established in \cite[Lem. 2.2]{sims-williams2012} (and independently in \cite{paravicini}). 
Given $u \in \go$, we can define a Radon measure $\sigma_Z^u$ on $Z$ by 
\begin{equation}
\label{eq:ZFiberRadon}
	\sigma_Z^u(\varphi) = \int_H \varphi(z \cdot \eta) \, d\beta^{s(z)}(\eta)
\end{equation}
Here $z \in r^{-1}_Z(u)$, so $\supp(\sigma_Z^u) = z \cdot H = r_Z^{-1}(u)$. (Of course one needs to know that this definition is independent of 
the choice of $z$.) 
One can define a family of Radon measures $\{ \sigma_{\Zop}^v \}_{v \in \go}$ 
on $\Zop$ with $\supp \sigma_{\Zop}^v = r_{\Zop}^{-1}(v)$ in a similar fashion.
Now given $w \in L$, define
\[
		\kappa^w(F) = \begin{cases}
			\lambda^w(F \vert_G) + \sigma_Z^w(F \vert_Z) & \text{if } w \in \go \\
			\sigma_{\Zop}^w(F \vert_{\Zop}) + \beta^w(F \vert_H) & \text{if } w \in \ho.
		\end{cases}	
\]
Then \cite[Lem. 2.2]{sims-williams2012} states that the family $\{\kappa^w\}$ defines a Haar system on $L$.


We now turn to the concept of exactness for groupoids, which is defined in terms of groupoid dynamical systems and crossed products.

\begin{defn}
	Let $\A$ be an upper semicontinuous $C^*$-bundle over $\go$. An \emph{action} of $G$ on $\A$ is a family $\alpha = 
	\{ \alpha_\gamma\}_{\gamma \in G}$, where:
	\begin{thmenum}
		\item $\alpha_\gamma : \A_{s(\gamma)} \to \A_{r(\gamma)}$ is an isomorphism for all $\gamma \in G$,
		\item if $(\gamma, \eta) \in G^{(2)}$, then $\alpha_{\gamma\eta} = \alpha_\gamma \circ \alpha_\eta$, and
		\item the assignment $(\gamma, a) \mapsto \gamma \cdot a = \alpha_\gamma(a)$ is continuous from $G * \A \to \A$.
	\end{thmenum}
	If $\alpha$ is an action of $G$ on $\A$, the triple $(\A, G, \alpha)$ is called a \emph{groupoid dynamical system}. We say that 
	$(\A, G, \alpha)$ is \emph{separable} if $A$ is separable and $G$ is second countable.
\end{defn}

All dynamical systems in this paper are assumed to be separable.
If $(\A, G, \alpha)$ is a groupoid dynamical system, the space $\Gamma_c(G, r^*\A)$ of continuous compactly supported sections becomes 
a $*$-algebra with respect to the product
\[
	f*g(\gamma) = \int_G f(\eta) \alpha_\eta \bigl( g(\eta^{-1}\gamma) \bigr) \, d\lambda^{r(\gamma)}(\eta)
\]
and involution
\[
	f^*(\gamma) = \alpha_\gamma \bigl( f(\gamma^{-1})^* \bigr)
\]
We equip $\Gamma_c(G, r^*\A)$ with a norm as follows: for $f \in \Gamma_c(G, r^*\A)$, define
\[
	\norm{f} = \sup \norm{\pi(f)},
\]
where $\pi$ ranges over all $*$-representations of $\gcra$ on Hilbert space that are continuous in the inductive limit topology. (Recall that a 
net $\{f_i\}$ converges to $f$ in the inductive limit topology if $f_i \to f$ uniformly and the sets $\supp(f_i)$ are eventually contained in a fixed 
compact set $K$.) We call this norm the \emph{universal norm}, and the completion of $\Gamma_c(G, r^*\A)$ with respect to the universal norm 
is called the \emph{(full) crossed product of $\A$ by $G$}, denoted $\A \rtimes_\alpha G$.

Exactness for groupoids is defined in terms of the \emph{reduced} crossed product, which can be constructed from the full crossed product via
induced representations. These representations (and many of the proofs in this paper) rely on the notion of an {equivalence} between dynamical 
systems, which is formally defined in \cite[Def. 5.1]{mw08}. In short, an \emph{equivalence} between two groupoid dynamical systems $(\A, G, \alpha)$ and 
$(\B, H, \beta)$ is an upper semicontinuous Banach bundle $p_\E : \E \to Z$ over a $(G, H)$-equivalence $Z$ together with 
\begin{itemize}
	\item $\A_{r(z)}-\B_{s(z)}$-imprimitivity bimodule structures on each fiber $\E_x$;
	\item commuting, continuous actions of $G$ and $H$ on the left and right of $\E$.
\end{itemize}
Additionally, the maps induced by the bimodule actions and inner products are continuous, and the bundle map $p_\E : \E \to Z$ is $G$- and $H$-equivariant.
Moreover, the groupoid actions must be compatible with the imprimitivity bimodule structure, and the $G$- and $\B$-actions on $\E$ commute, as do the $H$- 
and $\A$-actions. This definition is exactly what is needed to extend Renault's equivalence theorem 
to groupoid crossed products. Indeed, it is shown in \cite[Thm. 5.5]{mw08} that if $(\A, G, \alpha)$ and $(\B, H, \beta)$ are separable groupoid dynamical 
systems and $p : \E \to Z$ is an equivalence between $(\A, G, \alpha)$ and $(\B, H, \beta)$, then $X_0 = \Gamma_c(X, \E)$ completes to a 
$\A \rtimes_\alpha G - \B \rtimes_\beta H$-imprimitivity bimodule.
We will use this theorem in two particular special cases, which we describe below.

Our first example yields a ``Mackey machine'' for inducing representations of crossed products from 
closed subgroupoids. This construction is described in detail in Section 1 of \cite{goehle2010}. Let $(\A, G, \alpha)$ be a separable groupoid dynamical system, and let $H \subseteq G$ be a closed subgroupoid with Haar 
system. If we put $X = s^{-1}(\ho)$, then the natural right action of $H$ on $X$ is free and proper. Let $X^H$ denote the associated imprimitivity groupoid. 
It is noted in Section 2 of \cite{IonWil09} (and it follows from \cite[Prop. 5.2]{KMRW})
that $X^H$ can be equipped with a Haar system $\{\mu^{[\xi]}\}$, which is defined by
\[
	\int_{X^H} f \bigl( [\xi, \eta] \bigr) \, d\mu^{[\zeta]} \bigl([\xi, \eta] \bigr) = \int_G f \bigl( [\zeta, \eta] \bigr) \, d\lambda_{s(\zeta)}(\eta)
\]
for $f \in C_c(X^H)$. There is also a natural dynamical system associated to $X^H$. Consider the map $\rho : X/H \to \go$ given by
$\rho([\gamma]) = r(\gamma)$, and form the pullback bundle $\rho^*\A$ over $X/H$. Then $X^H$ acts on $\rho^*\A$ in a 
straightforward way: given $[\gamma, \eta] \in X^H$, define $\sigma_{[\gamma, \eta]} : \A_{r(\eta)} \to \A_{r(\gamma)}$ by
\[
	\sigma_{[\gamma, \eta]}(a) = \alpha_{\gamma \eta^{-1}}(a)
\]
for $a \in \A_{r(\eta)}$. It is shown in \cite[Prop. 2.3]{goehle2010} that $\sigma = \{\sigma_{[\gamma, \eta]}\}$ defines a continuous
action of $X^H$ on $\rho^*\A$.
%
On the other hand, the restriction $\A\vert_\ho$ is an upper semicontinuous $C^*$-bundle over $\ho$, so we have a dynamical system 
$(\A\vert_\ho, H, \alpha\vert_H)$. The dynamical systems $(\rho^*\A, X^H, \sigma)$ and $(\A\vert_\ho, H, \alpha\vert_H)$ are equivalent 
via $\E = (s^*\A)\vert_X$ by \cite[Prop. 6.3]{geoff}, yielding the following fact from \cite{goehle2010}.

\begin{prop}[{\cite[Prop. 2.4]{goehle2010}}]
\label{prop:imprimitivityDS}
	Let $(\A, G, \alpha)$ and $(\rho^*\A, X^H, \sigma)$ be as above. Then $\Zme_0 = \Gamma_c(X, s^*\A)$ is a $\rho^*\A 
	\rtimes_\sigma X^H - \A \vert_{\ho} \rtimes_{\alpha \vert_H} H$-pre-imprimitivity bimodule with respect to the following operations for 
	$f \in \Gamma_c(X^H, r^*(\rho^*\A))$, $g \in \Gamma_c(H, r^*\A)$, and $z, w \in \Zme_0$:
	\begin{align*}
		f \cdot z(\gamma) &= \int_G \alpha_\gamma^{-1} \bigl( f([\gamma, \eta]) \bigr) z(\eta) \, d\lambda_{s(\gamma)}(\eta) \\
		z \cdot g(\gamma) &= \int_H \alpha_\eta \bigl( z(\gamma \eta) g(\eta^{-1}) \bigr) \, d\lambda_H^{s(\gamma)}(\eta) \\
		\hip{z}{w}_{\A \rtimes H}(\eta) &= \int_G z(\xi \eta^{-1})^* \alpha_\eta \bigl( w(\xi) \bigr) \, d\lambda_{s(\eta)}(\xi) \\
		{_{\rho^*\A \rtimes X^H} \hip{z}{w}}([\gamma, \eta]) &= \int_H \alpha_{\gamma \xi} \bigl( z(\gamma \xi) w(\eta \xi)^* \bigr) \, 
			d\lambda_H^{s(\gamma)}(\xi)
	\end{align*}
	The completion $\Zme_H^G$ of $\Zme_0$ is a $\rho^*\A \rtimes_\sigma X^H - \A\vert_{\ho} \rtimes_{\alpha \vert_H} H$-imprimitivity 
	bimodule, and $\rho^*\A \rtimes_\sigma X^H$ and $\A \vert_\ho \rtimes_{\alpha \vert_H} H$ are Morita equivalent.
\end{prop}

The key to forming induced representations is the observation that $\A \rtimes_\alpha G$ acts nondegenerately on $\Zme_H^G$ via adjointable 
operators, with the action given by
\[
	f \cdot z(\gamma) = \int_G \alpha_\gamma^{-1} \bigl( f(\eta) \bigr) z(\eta^{-1}\gamma) \, d\lambda^{r(\gamma)}(\eta)
\]
for $f \in \Gamma_c(G, r^*\A)$ and $z \in \Zme_0$ \cite[Prop. 2.5]{goehle2010}. We can then use standard Rieffel induction techniques to construct
an induced representation $\Ind_H^G \pi$ of $\A \rtimes_\alpha G$ 
on the Hilbert space $\Zme_H^G \otimes \Hil$,
which is characterized on elementary tensors by
\[
	\Ind_H^G \pi(f)(z \otimes h) = f \cdot z \otimes h
\]
for $f \in \Gamma_c(G, r^*\A)$, $z \in \Zme_0$, and $h \in \Hil$.

\begin{exmp}
\label{exmp:RegularReps}
	We will implement the above discussion almost exclusively with $H = \go$. 
	Let $(\A, G, \alpha)$ be a separable groupoid dynamical system, and take $H = \go$. Then 
	In this case, $\A \vert_{\go}= \A$ and the action of $\go$ on $\A$ is trivial, so it is not hard to 
	check that $\A \rtimes \go = A$. Furthermore, $X = G$ and $X^H = G *_s G$. The bimodule $\Zme_{\go}^G$ for the induction process is a completion of $\Gamma_c(G, s^*\A)$ with respect to the
	following simplified operations: for $f \in \Gamma_c(G *_s G, r^*(\rho^*\A))$, $g \in \Gamma_c(\go, \A)$, and $z, w \in \Zme_0$,
	\begin{align*}
		f \cdot z(\gamma) &= \int_G \alpha_\gamma^{-1} \bigl( f(\gamma, \eta) \bigr) z(\eta) \, d\lambda_{s(\gamma)}(\eta) \\
		z \cdot g(\gamma) &= z(\gamma) g(s(\gamma)) \\
		\hip{z}{w}_A(u) &= \int_G z(\xi)^* w(\xi) \, d\lambda_u(\xi) \\
		{_{\rho^*\A \rtimes X^H} \hip{z}{w}}(\gamma, \eta) &= \alpha_\gamma \bigl( z(\gamma) w(\eta)^* \bigr).
	\end{align*}
	Now $\Ind_{\go}^G \pi$ acts on $\Zme \otimes_A \Hil$, 
	and it takes the usual form guaranteed above: for $f \in \gcra$, $z \in \Gamma_c(G,
	s^*\A)$, and $h \in \Hil$, $\Ind_{\go}^G \pi(f)(z \otimes h) = f \cdot z \otimes h$, where
	\[
		f \cdot z(\gamma) = \int_G \alpha_\gamma^{-1} \bigl(f(\eta)\bigr) z(\eta^{-1}\gamma) \, d\lambda^{r(\gamma)}(\eta).
	\]
	We will often refer to representations induced from $\go$ as \emph{regular representations}.
If we take $\pi$ to be faithful, then we can define the \emph{reduced norm} on $\Gamma_c(G, r^*\A)$:
\[
	\norm{f}_r = \norm{\Ind \pi(f)}.
\]
The resulting completion is the \emph{reduced crossed product}, denoted by $\A \rtimes_{\alpha, r} G$. 
\end{exmp}

We are mainly interested in reduced crossed
products, and it is critical
that Proposition \ref{prop:imprimitivityDS} descends to the level of 
reduced crossed products. This follows as a special case of \cite[Thm. 14]{sims-williams2013}, which is a version of Renault's Equivalence Theorem for
Fell bundle $C^*$-algebras.

\begin{cor}
\label{cor:ReducedImprimitivityDS}
	The operations defined in Proposition \ref{prop:imprimitivityDS} also make $\Zme_0$ into a $\rho^*\A \rtimes_{\sigma,r} X^H - \A \vert_{\ho} 
	\rtimes_{\alpha \vert_H,r} H$-pre-imprimitivity bimodule. Consequently, $\Zme_0$ completes to a $\rho^*\A \rtimes_{\sigma,r} X^H - \A 
	\vert_{\ho} \rtimes_{\alpha \vert_H,r} H$-imprimitivity bimodule $\Zme_{H,r}^G$, and $\rho^*\A \rtimes_{\sigma, r} X^H$ and $\A \vert_\ho 
	\rtimes_{\alpha\vert_H, r} H$ are Morita equivalent. Moreover, the kernels of the canonical quotient maps $\rho^*\A \rtimes_{\sigma} X^H
	\to \rhoacrossgh$ and $\A \vert_\ho \rtimes_{\alpha\vert_H} H \to \acrossh$ are matched up under the Rieffel correspondence from
	Corollary \ref{prop:imprimitivityDS}, and $\Zme_{H,r}^G$ is the quotient of $\Zme_H^G$ by the corresponding closed sub-bimodule.
\end{cor}

The reduced crossed product is defined more easily than the full one, but it can be poorly 
behaved at times. In particular, it is well known that reduced crossed products can fail to preserve short exact sequences, even for groups. As a
result, we have the following definition from Kirchberg and Wassermann: a locally compact group $G$ is \emph{exact} if whenever 
$(A, G, \alpha)$ is a dynamical system and $I$ is a $G$-invariant ideal in $A$, the sequence
\[
	\xymatrix{
		0 \ar[r] & I \rtimes_{\alpha \vert_I, r} G \ar[r] & A \rtimes_{\alpha, r} G \ar[r] & A/I \rtimes_{\alpha^I, r} G \ar[r] & 0 
	}
\]
of reduced crossed products is exact. This notion generalizes easily to groupoids.
Let $(\A, G, \alpha)$ be a separable groupoid dynamical system, and suppose $I$ is an ideal in $A$. Then $I$ and $A/I$ are both
$C_0(\go)$-algebras by \cite[\S 3.3]{lalonde2014}, and we denote the associated upper semicontinuous $C^*$-bundles by $\I$ and $\A/\I$, respectively.

\begin{defn}
	The ideal $I$ is \emph{$G$-invariant} if $\alpha_\gamma \bigl( \I_{s(\gamma)} \bigr) = \I_{r(\gamma)}$ for all $\gamma \in G$. 
\end{defn}

Note that this definition implies that the restriction $\alpha \vert_I = \{ \alpha_\gamma \vert_{\I_{s(\gamma)}} \}_{\gamma \in G}$ 
yields an action of $G$ on $\I$. Furthermore, for each $\gamma \in G$ we get an isomorphism
\[
	\alpha_\gamma^I : (\A/\I)_{s(\gamma)} \to (\A/\I)_{r(\gamma)}.
\]
Under the natural identification of $(\A/\I)_u$ with $\A_u/\I_u$, this action is just 
\[
	\alpha_\gamma^I \bigl( a(s(\gamma)) + \I_{s(\gamma)} \bigr) = \alpha_\gamma \bigl( a(s(\gamma)) \bigr) + \I_{r(\gamma)}.
\]
In particular, this says that the maps $\iota : I \to A$ and $q : A \to A/I$ are $G$-equivariant. Therefore, \cite[Prop. 6.3]{lalonde2014} guarantees that they 
yield maps $\iota \rtimes \id : \I \rtimes_{\alpha \vert_I} G \to \A \rtimes_\alpha G$ and $q \rtimes \id : \A \rtimes_\alpha G \to (\A/\I) \rtimes_{\alpha^I} G$. 
Furthermore, it is shown in \cite[Lem. 6.3.2]{ananth-renault} that the sequence
\begin{equation}
\label{eq:invariantsequence}
	\xymatrix{
		0 \ar[r] & \I \rtimes_{\alpha \vert_I} G \ar[r]^{\iota \rtimes \id} & \A \rtimes_\alpha G \ar[r]^(0.45){q \rtimes \id} & \A/\I \rtimes_{\alpha^I} 
		G \ar[r] & 0.
	}
\end{equation}
is exact. Alternatively, this fact follows from \cite[Thm. 3.7]{dana-marius}, which is a more general statement about Fell bundle $C^*$-algebras. It is
also shown in \cite[Prop. 6.10]{lalonde2014} that $\iota$ and $q$ induce maps at the level of reduced crossed products, so we get a sequence
\begin{equation}
\label{eq:ReducedExactSequence}
		\xymatrix{
		0 \ar[r] & \I \rtimes_{\alpha \vert_I, r} G \ar[r]^{\iota \rtimes \id} & \A \rtimes_{\alpha, r} G \ar[r]^(0.43){q \rtimes \id} & 
			\A/\I \rtimes_{\alpha^I, r} G \ar[r] & 0.
		}	
\end{equation}
However, \eqref{eq:ReducedExactSequence} is not exact in general. Gromov has famously produced examples of \emph{groups} for 
which \eqref{eq:ReducedExactSequence} fails to be exact, and there are more tractable examples of groupoids for which 
\eqref{eq:ReducedExactSequence} is not exact. Given the unfortunate existence of such groupoids, it makes sense to single out the ones for 
which \eqref{eq:ReducedExactSequence} is always exact. 

\begin{defn}
	A second countable locally compact groupoid $G$ is said to be \emph{exact} if whenever $(\A, G, \alpha)$ is a separable 
	groupoid dynamical system and $I$ is a $G$-invariant ideal in $A$, the sequence
	\begin{equation}
	\label{eq:invariantsequence2}
		\xymatrix{
			0 \ar[r] & \I \rtimes_{\alpha \vert_I, r} G \ar[r]^{\iota \rtimes \id} & \A \rtimes_{\alpha, r} G \ar[r]^(0.45){q \rtimes \id} & 
				\A/\I \rtimes_{\alpha^I, r} G \ar[r] & 0.
		}
	\end{equation}
	is short exact.
\end{defn}

Note that we have the following immediate corollary to \cite[Lem. 6.3.2]{ananth-renault} when $G$ is measurewise amenable.

\begin{cor}
	If $G$ is a measurewise amenable secound countable locally compact Hausdorff groupoid, then $G$ is exact.
\end{cor}

\section{The Main Technical Result}
The proof of the main theorem relies on a variation of \cite[Thm. 3.9]{kw99-2} for groupoids. This section is devoted to proving this fairly technical fact.

Let $G$ be a locally compact Hausdorff groupoid with Haar system $\{\lambda_G^u\}_{u \in \go}$, and let $(\A, G, \alpha)$ be a dynamical
system. Suppose $H$ is a closed subgroupoid of $G$ with Haar system $\{\lambda_H^u\}_{u \in \ho}$, and put $X = s^{-1}(\ho)$.
Then the dynamical system $(\rho^*\A, G^H, \sigma)$ is equivalent to $(\A \vert_\ho, H, \alpha\vert_H)$, so the reduced crossed products 
$\rho^*\A \rtimes_{\sigma, r} X^H$ and $\A \vert_{\ho} \rtimes_{\alpha \vert_H, r} H$ are Morita equivalent by Corollary \ref{cor:ReducedImprimitivityDS}.
If $I \subseteq A$ is a $G$-invariant ideal,
then we have dynamical systems 
$(\I, G, \alpha \vert_I)$ and $(\A/\I, G, \alpha^I)$, and the restrictions to $H$ yield dynamical systems $(\I \vert_\ho, H, \alpha \vert_I)$ 
and $(\A/\I \vert_\ho, H, \alpha^I)$. We need to relate these to the corresponding pullback dynamical systems.
We begin with the following generalization of \cite[Lemma 3.8]{kw99-2}.
	
\begin{lem}
\label{lem:kw3.8}
	Let $\iota : I \to A$ denote the inclusion map and $q : A \to A/I$ the quotient map. Then the sequence
	\begin{equation}
	\label{eq:pullbacksequence}
	\xymatrix{
		0 \ar[r] & \rho^*I \ar[r]^(0.48){\rho^*\iota} & \rho^*A \ar[r]^(0.4){\rho^*q} & \rho^*(A/I) \ar[r] & 0
	}
	\end{equation}
	is exact.
\end{lem}
\begin{proof}
	By \cite[Prop. 3.7]{lalonde2014}, the map $\rho^*\iota$ is given by
	\[
		\rho^*\iota(f)([\gamma]) = \iota_{\rho([\gamma])} ( f([\gamma]) ) = \iota_{r(\gamma)} (f([\gamma]) )
	\]
	for $f \in \rho^*I$ and $\gamma \in X$. Therefore, $\rho^*\iota(f) = 0$ if and only if $\iota_{r(\gamma)}(f([\gamma])) = 0$ for all
	$\gamma \in X$. This in turn holds if and only if $f([\gamma]) = 0$ for all $\gamma \in X$, since $\iota_{r(\gamma)}$ is just the 
	inclusion of the ideal $\I_{r(\gamma)}$ into $\A_{r(\gamma)}$. Therefore, $\rho^*\iota$ is injective.
	
	Now recall that $\rho^*(A/I) \cong C_0(X/H) \otimes_{C_0(\go)} A/I$, where we identify $f \otimes a$ 
	with the section $\gamma \cdot H \mapsto f([\gamma]) a(\rho([\gamma]))$. Hence the set
	\[
		C_c(X/H) \odot A/I = \spa\{ f \otimes a : f \in C_c(X/H), a \in A/I\},
	\]
	is dense in $\rho^*(A/I)$. 
	Let $f \in C_c(X/H)$ and $a \in A/I$, 
	and pick $\tilde{a} \in A$ with $q(\tilde{a}) = a$. Then $f \otimes \tilde{a} \in C_c(X/H) \odot A \subseteq \rho^*A$, so
	\begin{align*}
		\rho^*q(f \otimes \tilde{a})([\gamma]) &= q_{\rho([\gamma])} \bigl( f \otimes \tilde{a}([\gamma]) \bigr) \\
			&= f([\gamma]) q_{r(\gamma)} \bigl( \tilde{a}(r(\gamma)) \bigr) \\
			&= f([\gamma]) a(r(\gamma)) \\
			&= f \otimes a([\gamma])
	\end{align*}
	for all $\gamma \in X$. Therefore, $\rho^*q(f \otimes \tilde{a}) = f \otimes a$, so $\rho^*q$ maps onto the dense subspace $C_c(X/H) \odot A/I$ 
	of $\rho^*(A/I)$. It follows that $\rho^*q$ is surjective.
	
	Finally, we check that $\image(\rho^*\iota) = \ker(\rho^*q)$. Certainly we have $\rho^*q \circ \rho^*\iota = 0$: 
	\[
		\rho^*q(\rho^*\iota(f))([\gamma]) = q_{r(\gamma)} \bigl( \rho^*\iota(f)([\gamma]) \bigr) = 
			q_{r(\gamma)} \bigl( \iota_{r(\gamma)}(f([\gamma])) \bigr) = 0
	\]
	for $f \in \rho^*I$, since $q \circ \iota = 0$. Now suppose $g \in \ker(\rho^*q)$. Then for all $\gamma \in X$ we have
	\[
		\rho^*q(g)([\gamma]) = q_{r(\gamma)} ( g([\gamma]) ) = 0,
	\]
	so $g([\gamma]) \in \ker q_{r(\gamma)} = \image \iota_{r(\gamma)} = \I_{r(\gamma)}$. Therefore, if we define a section 
	$\tilde{g} \in \rho^*I$ by $\tilde{g}([\gamma]) = g([\gamma])$, then $\rho^*\iota(\tilde{g}) = g$. Hence $g \in 
	\image(\rho^*\iota)$, and (\ref{eq:pullbacksequence}) is exact.
\end{proof}
	
The rest of this section deals with modifying \cite[Thm. 3.9]{kw99-2} for groupoids. This requires us to consider the effect of the functor 
$\underline{\phantom{\A}} \rtimes_{\sigma, r} X^H$ on the sequence (\ref{eq:pullbacksequence}). For this to make sense, we first need
to know that $\rho^*I$ is $X^H$-invariant.
	
\begin{prop}
\label{prop:XHInvariantIdeal}
	The ideal $\rho^*I$ is invariant under the $X^H$-action on $\rho^*A$.
\end{prop}
\begin{proof}
	We need to check that if $a \in \rho^*\I_{s([\gamma, \eta])}$, then $\sigma_{[\gamma, \eta]}(a) \in \rho^*\I_{r([\gamma, \eta])}$. Well,
	$\rho^*\I_{s([\gamma, \eta])} = \rho^*\I_{[\eta]}$, which is naturally identified with $\I_{r(\eta)}$. Then
	\[
		\sigma_{[\gamma, \eta]}(a) = \alpha_{\gamma \eta^{-1}}(a) \in \I_{r(\gamma\eta^{-1})} = \I_{r(\gamma)},
	\]
	since $I \subseteq A$ is $\alpha$-invariant. But $\rho^*\I_{r([\gamma, \eta])} = \I_{r(\gamma)}$, so $\sigma_{[\gamma, \eta]}(a) \in 
	\rho^*\I_{r([\gamma, \eta])}$. In fact, $\sigma_{[\gamma, \eta]}$ restricts to an isomorphism of $\rho^*\I_{s([\gamma, \eta])}$ onto
	$\rho^*\I_{r([\gamma, \eta])}$, since $\alpha_{\gamma \eta^{-1}} : \I_{r(\eta)} \to I_{r(\gamma)}$ is an isomorphism.
\end{proof}
	
Proposition \ref{prop:XHInvariantIdeal} actually buys us a little more. The fact that $\rho^*I$ is an
$X^H$-invariant ideal is equivalent to saying that the inclusion $\rho^*\iota$ is an $X^H$-equivariant homomorphism. Therefore, 
\cite[Prop. 6.10]{lalonde2014} guarantees that there is a homomorphism
\[
	\rho^*\iota \rtimes \id : \rho^*\I \rtimes_{\sigma \vert_I, r} X^H \to \rho^*\A \rtimes_{\sigma, r} X^H.
\]
Similarly, the quotient map $\rho^*q : \rho^*A \to \rho^*(A/I)$ is $X^H$-equivariant, so we have a homomorphism
\[
	\rho^*q \rtimes \id : \rho^*\A \rtimes_{\sigma, r} X^H \to \rho^*(\A/\I) \rtimes_{\alpha^I, r} X^H.
\]
With these ideas in place, we can now state our version of \cite[Thm. 3.9]{kw99-2}. For simplicity, we will write crossed products of the
form $\A\vert_{\ho} \rtimes_{\alpha\vert_, r} H$ as $\A \rtimes_{\alpha, r} H$, with it understood that we are restricting the bundles and actions
to $H$.
	
\begin{thm}
\label{thm:kwexact}
	The sequence
	\[
		\xymatrix{
			0 \ar[r] & \icrossh \ar[r]^{\iota \rtimes \id} & \acrossh \ar[r]^(0.43){q \rtimes \id} & \amodicrossh \ar[r] & 0
		}
	\]
	is exact if and only if 
	\[
		\xymatrix@C=0.5cm{
			0 \ar[r] & \rhoicrossgh \ar[rr]^(0.52){\rho^*\iota \rtimes \id} & & \rhoacrossgh 
				\ar[rr]^(0.42){\rho^*q \rtimes \id} & & \rhoamodicrossgh \ar[r] & 0
		}
	\]
	is exact.
\end{thm}
	
The proof requires a series of lemmas. To simplify notation, we use $\BI$ and $\EI$ to denote the ideals $\icrossh$ and $\rhoicrossgh$ of 
$\acrossh$ and $\rhoacrossgh$, respectively. Similarly, we let $\J = \ker(q \rtimes \id)$ and $\Jtwiddle = \ker(\rho^*q \rtimes \id)$. We clearly 
have $\BI \subseteq \J$ and $\EI \subseteq \Jtwiddle$, and we know from Corollary \ref{cor:ReducedImprimitivityDS} that $\rhoacrossgh$ 
is Morita equivalent to $\acrossh$. Therefore, our goal is to show that $\BI$ and $\EI$ are matched up under the Rieffel correspondence, and 
likewise for $\J$ and $\Jtwiddle$. \index{Rieffel correspondence}
	
Note that Corollary \ref{cor:ReducedImprimitivityDS} also implies that $\rhoicrossgh$ and $\icrossh$ are Morita equivalent via a 
completion of $\Y_0 = \Gamma_c(X, s^*\I)$. Also, it is clear that the set $\Y_0$ embeds naturally into $\Zme_0 = \Gamma_c(X, s^*\A)$. 
We show first that this embedding extends to an embedding of the $\rhoicrossgh - \icrossh$-imprimitivity bimodule $\Y_r = \overline{\Y}_0$ 
into the $\rhoacrossgh - \acrossh$-imprimitivity bimodule $\Zme_r = \overline{\Zme}_0$.

\begin{lem}
\label{lem:one}
	The natural embedding of $\Y_0$ into $\Zme_0$ is isometric with respect to the reduced norms (i.e., the norms coming from 
	$\icrossh$ and $\acrossh$, respectively), and therefore extends to an embedding of $\Y_r$ into $\Zme_r$.
\end{lem}
\begin{proof}
	Let $z, w \in \Y_0$. Then for $\eta \in H$ we have
	\begin{align*}
		\hip{z}{w}_{\icrossh}(\eta) &= \int_G z(\xi \eta^{-1})^* \alpha_\eta(w(\xi)) \, d\lambda_{s(\eta)}(\xi),
	\end{align*}
	which is precisely the formula for $\hip{z}{w}_{\acrossh}(\eta)$ obtained from viewing $z$ and $w$ as elements of $\Gamma_c(X, s^*\A)$.
	Since $\Gamma_c(H, r^*\I)$ embeds isometrically into $\Gamma_c(H, r^*\A)$ with respect to the reduced norm, it follows that the 
	embedding $\Y_0 \hookrightarrow \Zme_0$ is isometric.
\end{proof}
	
	\begin{lem}
		If we view $\Y_r$ as a closed sub-bimodule of $\Zme_r$ via Lemma \ref{lem:one}, we have $\Y_r = \overline{\Zme_r \cdot \BI}$. 
		Consequently, $\Y_r$ is the submodule assigned to $\BI$ under the Rieffel correspondence.
	\end{lem}
	\begin{proof}
		Let $z \in \Zme_0$ and $g \in \Gamma_c(H, r^*\I) \subseteq \icrossh$. Then for $\gamma \in X$ we have
		\begin{equation}
		\label{eq:yraction}
			z \cdot g(\gamma) = \int_H \alpha_\eta\bigl( z(\gamma\eta)g(\eta^{-1}) \bigr) \, d\lambda_H^{s(\gamma)}(\eta).
		\end{equation}
		Note that $z(\gamma \eta) \in s^*\A_{\gamma \eta} = \A_{s(\eta)}$ and $g(\eta^{-1}) \in \I_{s(\eta)} \subseteq \A_{s(\eta)}$, so 
		$z(\gamma \eta)g(\eta^{-1})$ belongs to the ideal $\I_{s(\eta)}$ of $\A_{s(\eta)}$. It follows that the integrand in (\ref{eq:yraction}) 
		belongs to $\I_{r(\eta)} = \I_{s(\gamma)}$, since $I$ is an invariant ideal in $A$. Therefore, $z \cdot g(\gamma) \in \I_{s(\gamma)} = s^*\I_\gamma$, 
		so $z \cdot g \in \Gamma_c(X, s^*\I)$. It follows that $\Zme_0 \cdot \BI \subseteq \Y_0$, so $\overline{\Zme_r \cdot \BI} \subseteq \Y_r$.
		
		On the other hand, the computations above (together with those of Lemma \ref{lem:one}) show that the inclusion $\Y_r 
		\hookrightarrow \Zme_r$ is an embedding of right Hilbert $\BI$-modules. Since $\Y_r$ is a $\EI$-$\BI$-imprimitivity bimodule, 
		we have
		\[
			\Y_r = \overline{\Y_r \cdot \BI} \subseteq \overline{\Zme_r \cdot \BI}.
		\]
		Therefore, $\Y_r = \overline{\Zme_r \cdot \BI}$, and $\Y_r$ is the submodule of $\Zme_r$ associated to $\BI$ under the Rieffel 
		correspondence.
	\end{proof}
	
	\begin{lem}
		We have $\Y_r = \overline{\EI \cdot \Zme_r}$, so $\EI$ and $\Y_r$ are matched up under the Rieffel correspondence.
	\end{lem}
	\begin{proof}
		The proof is similar to that of the last lemma. Note that if $z \in \Zme_0$ and $f \in \Gamma_c(X^H, r^*(\rho^*\I)) \subseteq \EI$, then
		\begin{equation}
		\label{eq:eiaction}
			f \cdot z(\gamma) = \int_G \alpha_\gamma^{-1} \bigl( f([\gamma, \eta]) \bigr) z(\eta) \, d\lambda_{s(\gamma)}(\eta).
		\end{equation}
		Since $f([\gamma, \eta]) \in r^*(\rho^*\I)_{[\gamma, \eta]} = \rho^*\I_{[\gamma]} = \I_{r(\gamma)}$ and $z(\eta) \in s^*\A_\eta = 
		\A_{s(\eta)} = \A_{s(\gamma)}$, the integrand in (\ref{eq:eiaction}) belongs to $\I_{s(\gamma)}$. Thus $f \cdot z \in \Y_0$. In particular, 
		this shows that $\overline{\EI \cdot \Zme_r} \subseteq \Y_r$. On the other hand, taking $z \in \Y_0$ in (\ref{eq:eiaction}) yields the 
		same formula, so $\Y_r$ embeds into $\Zme_r$ as a right Hilbert $\EI$-submodule. Thus $\Y_r = \overline{\EI \cdot \Y_r} \subseteq 
		\overline{\EI \cdot \Zme_r}$, so $\Y_r = \overline{\EI \cdot \Zme_r}$. It follows that $\Y_r$ and $\EI$ are matched up under the Rieffel 
		correspondence.
	\end{proof}
	
By combining the previous three lemmas, we immediately obtain one half of our desired result:
\begin{prop}
	Under the the Morita equivalence of $\rhoacrossgh$ and $\acrossh$, the ideals $\rhoicrossgh$ and $\icrossh$ are paired by the Rieffel correspondence.
\end{prop}

We now turn our attention to the ideals $\Jtwiddle = \ker(\rho^*q \rtimes \id)$ and $\J = \ker(q \rtimes \id)$. Let $\pi : (A/I)(\ho) \to B(\Hil)$ be a 
faithful separable representation. (Here we use $(A/I)(\ho)$ to denote the section algebra of 
the restricted bundle $\A \vert_{\ho}$.) We will frequently use the fact that $\pi$ can be viewed as a representation on $L^2(\ho*\frah, \mu)$ for some analytic Borel Hilbert
bundle $\ho * \frah$ and finite Borel measure $\mu$ on $\ho$. Thus $\pi$ has a direct integral decomposition
\[
	\pi = \int^\oplus_{\ho} \pi_u \, d\mu(u),
\]
where $\pi_u$ is a representation of $(\A/\I)_u$ on $\Hil(u)$.

Now put $\tilde{\pi} = \pi \circ q$. Then $\tilde{\pi}$ is a representation of $A(\ho)$ with kernel $I(\ho)$.  We can form the regular representation $\indhpi$ of $\amodicrossh$, which is faithful and acts on 
$\Zme^{A/I} \otimes_{A/I} \Hil$, where
\[
	\Zme^{A/I} = \overline{\Zme_0^{A/I}} = \overline{\Gamma_c(X, s^*(\A/\I))}
\]
is the usual $\fullrhoicrossgh - \fullicrossh$-imprimitivity bimodule. Then $(\indhpi) \circ (q \rtimes \id)$ is a representation of $\acrossh$ on 
$\Zme^{A/I} \otimes_{A/I} \Hil$ with kernel $\J$. On the other hand, we could form the representation $\indhpitilde$ of $\acrossh$, which acts on 
$\Zme \otimes_A \Hil$. Our goal then is to show that $\indhpitilde$ and $(\indhpi) \circ (q \rtimes \id)$ are unitarily equivalent, so that 
$\ker(\indhpitilde) = \J$.

\begin{lem}
\label{lem:UnitaryU}
	There is a unitary $U : \Zme \otimes_A \Hil \to \Zme^{A/I} \otimes_{A/I} \Hil$ characterized by
	\[
		U(f \otimes h) = s^*q(f) \otimes h
	\]
	for $f \in \Gamma_c(H, s^*\A)$ and $h \in \Hil$. Moreover, $U$ intertwines $\indhpitilde$ and $\bigl( \indhpi \bigr) \circ (q \rtimes \id)$.
\end{lem}
\begin{proof}
	Recall that the inner product on $\Zme^{A/I} \otimes_{A/I} \Hil$ is characterized by
	\[
		\ip{f \otimes h}{g \otimes k} = \ip{\pi(\hip{g}{f}_{A/I})h}{k}
	\]
	for $f, g \in \Gamma_c(H, s^*(\A/\I))$ and $h, k \in \Hil$, where
	\[
		\hip{g}{f}_{A/I}(u) = \int_H g(\eta)^*f(\eta) \, d\lambda_u(\eta)
	\]
	for $u \in \ho$. Similarly, the inner product on $\Zme \otimes_{A} \Hil$ satisfies
	\[
		\ip{f \otimes h}{g \otimes k} = \ip{\tilde{\pi}(\hip{g}{f}_A)h}{k}
	\]
	for $f, g \in \Gamma_c(H, s^*\A)$ and $h, k \in \Hil$, with
	\[
		\hip{g}{f}_A(u) = \int_H g(\eta)^*f(\eta) \, d\lambda_u(\eta)
	\]
	for $u \in \ho$. Let $f \otimes h, g \otimes k \in \Gamma_c(H, s^*\A) \odot \Hil$. Then we have
	\begin{align*}
		&\ip{U(f \otimes h)}{U(g \otimes k)} = 
			 \ip{\pi\bigl( \hip{s^*q(g)}{s^*q(f)}_{A/I}\bigr) h}{k} \\
			&\quad \quad= \int_\ho  \ip{\pi_u\bigl(\hip{s^*q(g)}{s^*q(f)}_{A/I}(u)\bigr) h(u)}{k(u)} \, d\mu(u) \\
			&\quad \quad= \int_\ho \ip{\pi_u \biggl( \int_H \bigl(s^*q(g)(\eta)\bigr)^* \bigl(s^*q(f)(\eta)\bigr) \, d\lambda_u(\eta) \biggr) h(u)}{k(u)} \, d\mu(u) \\
			&\quad \quad= \int_\ho \int_H \ip{\pi_u(q_{s(\eta)}(g(\eta))^* q_{s(\eta)}(f(\eta))) h(u)}{k(u)} \, d\lambda_u(\eta) d\mu(u) \\
			&\quad \quad= \int_H \ip{\pi_u \circ q_u(g(\eta)^*f(\eta)) h(u)}{k(u)} \, d\nu^{-1}(\eta).
	\end{align*}
	But $\pi_u \circ q_u = (\pi \circ q)_u = \tilde{\pi}_u$, so the above equation becomes
	\[
		\int_H \ip{\tilde{\pi}_{s(\eta)}(g(\eta)^*f(\eta))h(s(\eta))}{k(s(\eta))} \, d\nu^{-1}(\eta) = \ip{f \otimes h}{g \otimes k}
	\]
	in $\Zme \otimes_A \Hil$. It is easy to see that $U: \Zme_0 \odot \Hil \to \Zme_0^{A/I} \odot \Hil$ is surjective, so it is an isometry of $\Zme_0 \odot \Hil$ onto 
	$\Zme_0^{A/I} \odot \Hil$ that extends to a unitary $U : \Zme \otimes \Hil \to \Zme^{A/I} \otimes \Hil$.
	
	Now let $f \in \Gamma_c(H, r^*\A)$, $g \in \Gamma_c(H, s^*\A)$, and $h \in \Hil$. Then we have
		\begin{align*}
			\bigl( \indhpi \bigr) \circ (q \rtimes \id)(f) U(g \otimes h) &= \indhpi\bigl(q \rtimes \id(f) \bigr)(s^*q(g) \otimes h) \\
				&= q \rtimes \id(f) \cdot s^*q(g) \otimes h.
		\end{align*}
		Now observe that
		\begin{align*}
			q \rtimes \id(f) \cdot s^*q(g) &= \int_H \alpha_\gamma^{-1}\bigl( q \rtimes \id(f)(\eta) \bigr)s^*q(g)(\eta^{-1}\gamma) \, 
					d\lambda^{r(\gamma)}(\eta) \\
				&= \int_H \alpha_\gamma^{-1} \bigl( q_{r(\eta)}(f(\eta)) \bigr) q_{s(\eta^{-1}\gamma)}(g(\eta^{-1}\gamma)) \, d\lambda^{r(\gamma)}(\eta) \\
				&= \int_H q_{s(\gamma)} \bigl( \alpha_\gamma^{-1}(f(\eta)) g(\eta^{-1} \gamma) \bigr) \, d\lambda^{r(\gamma)}(\eta) \\
				&= q_{s(\gamma)} \biggl( \int_H \alpha_\gamma^{-1}(f(\eta)) g(\eta^{-1} \gamma) \, d\lambda^{r(\gamma)}(\eta) \biggr) \\
				&= s^*q(f \cdot g)(\gamma).
		\end{align*}
		Thus
		\begin{align*}
			\bigl( \indhpi \bigr) \circ (q \rtimes \id)(f) U(g \otimes h) &= s^*q(f \cdot g) \otimes h \\
				&= U(f \cdot g \otimes h) \\
				&= U \bigl( \indhpitilde(f)(g \otimes h) \bigr),
		\end{align*}
		so $U$ intertwines the two representations.
\end{proof}
	
	
Now let $\Zme_r^{A/I}$ denote the $\rhoamodicrossgh-\amodicrossh$-imprimitivity
bimodule obtained by completing $\Gamma_c(X, s^*(\A/\I))$.
We can use this module, along with $\Zme_r$, to induce representations
\[
	\tau = \Zme_r^{A/I} - {\rm{Ind}} \bigl( \indhpi \bigr)
\]
and
\[
	\tilde{\tau} = \Zme_r - {\rm{Ind}} \bigl( \indhpitilde \bigr)
\]
of $\rhoamodicrossgh$ and $\rhoacrossgh$, respectively. We aim to show that $\tau \circ (\rho^*q \rtimes \id)$ is unitarily equivalent to 
$\tilde{\pi}$, which will imply that $\tilde{\tau}$ has kernel $\Jtwiddle$. Note that $\tau$ acts on the Hilbert space
\[
	\Zme_r^{A/I} \otimes_{\A/\I \rtimes H} \bigl( \Zme^{A/I} \otimes_{A/I} \Hil) = \overline{\Gamma_c(X, s^*(\A/\I)) \odot \Zme^{A/I} \odot \Hil},
\]
while $\tilde{\tau}$ acts on
\[
	\Zme_r \otimes_{\A \rtimes H} \bigl(\Zme \otimes_A \Hil) = \overline{\Gamma_c(X, s^*\A) \odot \Zme \odot \Hil}.
\]
Therefore, we need to find a unitary $V : \Zme_r \otimes (\Zme \otimes \Hil) \to \Zme_r^{A/I} \otimes (\Zme^{A/I} \otimes \Hil)$ that 
intertwines $\tau$ and $\tilde{\tau}$.
	
\begin{lem}
\label{lem:tauintertwine}
	Define $V : \Gamma_c(X, s^*\A) \odot (\Zme \otimes_A \Hil) \to \Gamma_c(X, s^*(\A/\I)) \odot (\Zme^{A/I} \otimes_{A/I} \Hil)$ on
	elementary tensors by
	\[
		V(f \otimes z) = s^*q(f) \otimes U(z),
	\]
	where $U : \Zme \otimes_A \Hil \to \Zme^{A/I} \otimes_{A/I} \Hil$ is the unitary from Lemma \ref{lem:UnitaryU} Then $V$ extends 
	to a unitary $V : \Zme_r \otimes_{\A \rtimes H} (\Zme \otimes_A \Hil) \to \Zme_r^{A/I} \otimes_{\A/\I \rtimes H} (\Zme^{A/I} \otimes_{A/I} 
	\Hil)$ intertwining $\tilde{\tau}$ and $\tau$.
\end{lem}
\begin{proof}
	Let $f, g \in \Gamma_c(X, s^*\A)$ and $z, w \in \Zme \otimes_A \Hil$. Then
	\begin{align*}
		\ip{V(f \otimes z}{V(g \otimes w)} &= \ip{s^*q(f) \otimes U(z)}{s^*q(g) \otimes U(w)} \\
			&= \ip{\indhpi\bigl( \hip{s^*q(g)}{s^*q(f)}_{A/I \rtimes H} \bigr) U(z)}{U(w)}.
	\end{align*}
	Now note that
	\begin{align*}
		\hip{s^*q(g)}{s^*q(f)}_{A/I \rtimes H}(\eta) &= \int_G s^*q(g)(\xi \eta^{-1})^* \alpha_\eta^I\bigl(s^*q(f)(\xi) \bigr) \, d\lambda_{s(\eta)}(\xi) \\
			&= \int_G q_{s(\xi\eta^{-1})} \bigl( g(\xi\eta^{-1}) \bigr)^* \alpha_\eta^I \bigl( q_{s(\eta)}(f(\xi)) \bigr) \, d\lambda_{s(\eta)}(\xi) \\
			&= \int_G q_{r(\eta)} \bigl(g(\xi \eta^{-1})^* \bigr) q_{r(\eta)} \bigl(\alpha_\eta(f(\xi)) \bigr) \, d\lambda_{s(\eta)}(\xi) \\
			&= q_{r(\eta)} \bigl( \hip{g}{f}_{A \rtimes H}(\eta) \bigr) \\
			&= q \rtimes \id \bigl( \hip{g}{f}_{A \rtimes H} \bigr)(\eta).
	\end{align*}
	Therefore,
	\begin{align*}
		\ip{V(f \otimes z)}{V(g \otimes w)} &= \ip{\bigl( \indhpi \bigr) \circ (q \rtimes \id) \bigl( \hip{g}{f}_{A \rtimes H} \bigr) U(z)}{U(w)} \\
			&= \ip{U \cdot \indhpitilde \bigl( \hip{g}{f}_{A \rtimes H} \bigr) z}{U(w)} \\
			&= \ip{\indhpitilde \bigl( \hip{g}{f}_{A \rtimes H} \bigr)z}{w} \\
			&= \ip{f \otimes z}{g \otimes w}.
	\end{align*}
	Thus $V$ defines an isometry of $\Gamma_c(X, s^*\A) \odot (\Zme \otimes_A \Hil)$ onto $\Gamma_c(X, s^*(\A/\I)) \odot (\Zme^{A/I}
	\otimes_{A/I} \Hil)$, which then extends to a unitary $V : \Zme_r \otimes_{\A \rtimes H} (\Zme \otimes_A \Hil) \to \Zme_r^{A/I} 
	\otimes_{\A/\I \rtimes H} (\Zme^{A/I} \otimes_{A/I} \Hil)$.

	Now let $g \in \Gamma_c(G^H, r^*(\rho^*\A))$, $f \in \Gamma_c(X, s^*\A)$, and $z \in \Zme$. Then
	\begin{align*}
		\tau \circ (\rho^*q \rtimes \id)(g) V(f \otimes z) &= \tau \bigl( (\rho^*q \rtimes \id)(g) \bigr) (s^*q(f) \otimes U(z)) \\
			&= \bigl( (\rho^*q \rtimes \id)(g) \cdot q(f) \bigr) \otimes U(z).
	\end{align*}
	Observe that
	\begin{align*}
		(\rho^*q \rtimes \id)(q) \cdot s^*q(f)(\gamma) &= \int_G \alpha_\gamma^{-1} \bigl((\rho^*q \rtimes \id)(g)([\gamma, \eta]) \bigr) 
			s^*q(f)(\eta) \, d\lambda_{s(\gamma)}(\eta) \\
			&= \int_G \alpha_\gamma^{-1} \bigl( q_{r(\gamma)}(g([\gamma, \eta])) \bigr) q_{s(\gamma)}(f(\eta)) \, 
				d\lambda_{s(\gamma)}(\eta) \\
			&=  q_{s(\gamma)} \biggl( \int_G \alpha_\gamma^{-1} \bigl( g([\gamma, \eta]) \bigr) f(\eta) \, d\lambda_{s(\gamma)}(\eta) \biggr) \\
			&= s^*q(g \cdot f)(\gamma).
	\end{align*}
	Therefore,
	\begin{align*}
		\tau \circ (\rho^*q \rtimes \id)(g) V(f \otimes z) &= s^*q(g \cdot f) \otimes U(z) \\
			&= V(g \cdot f \otimes z) \\
			&= V\bigl( \tilde{\tau}(g)(f \otimes z) \bigr),
	\end{align*}
	so $V$ intertwines $\tau$ and $\tilde{\tau}$.
\end{proof}
	
\begin{proof}[Proof of Theorem \ref{thm:kwexact}]
	Lemma \ref{lem:tauintertwine} implies that $\ker(\tilde{\tau}) = \ker(\tau) = \Jtwiddle$. By definition, $\tilde{\tau}$ is induced from 
	$\indhpitilde$ via $\Zme_r$, and $\ker(\indhpitilde) = \J$. Therefore, the ideals $\Jtwiddle$ and $\J$ are matched up under the Rieffel 
	correspondence. Since $\EI$ and $\BI$ also correspond, it follows that $\EI = \Jtwiddle$ if and only if $\BI = \J$. That is, the first sequence 
	in Theorem \ref{thm:kwexact} is exact if and only if the second is exact.
\end{proof}
	
\section{Equivalence and Exactness}
With Theorem \ref{thm:kwexact} in hand, we can proceed with proving the main result. Let $G$ and $H$ be locally compact Hausdorff 
groupoids with Haar systems $\{\lambda^u\}_{u \in \go}$ and $\{\beta^v\}_{v \in \ho}$, $Z$ a $(G, H)$-equivalence, and $L$ the associated 
linking groupoid. 
We will use Theorem \ref{thm:kwexact} to show that if $G$ is exact, then $L$ is exact. We will then show that the exactness of $L$ descends to $H$, again 
using Theorem \ref{thm:kwexact}.

Let $(\B, L, \beta)$ be a separable dynamical system. We view $G$ as a closed subgroupoid of $L$ with Haar system, and we let 
$(\A, G, \alpha)$ denote the restriction of $(\B, L, \beta)$ to $G$, i.e., $\A = \B \vert_{\go}$ and $\alpha = \beta \vert_G$. If we define 
\[
	X = s^{-1}(\go) = G \sqcup \Zop,
\]
then we know that $G$ acts freely and properly on the right of $X$. If we let $X^G$ be the associated imprimitivity groupoid, then $X$ 
is an $(X^G, G)$-equivalence. As in the previous section, we define $\rho: X/G \to \lo$ by
$\rho([\gamma]) = r(\gamma)$.
We'll see that in this particular case, the imprimitivity groupoid $X^G$ can be identified naturally with $L$. 

\begin{prop}
	The space $X$ is an $(L, G)$-equivalence with respect to the obvious left and right actions of $L$ and $G$.
\end{prop}
\begin{proof}
	This is a special case of Example 5.33(7) from \cite{muhly}. Note that $\go \subseteq \lo$ is a closed subset that meets every orbit, and that $G$ is nothing
	more than the reduction $L \vert_{\go}$. Thus $X$ is an $(L, G)$-equivalence as long as the restrictions of $r$ and $s$ to $X$ are open. But this follows
	immediately from the fact that $r, s: L \to \lo$ are open and $X$ is open in $L$.
\end{proof}
	
%
%
	
	It is a well-known fact \cite[\S 2]{MRW} that any groupoid equivalence induces a natural isomorphism with the appropriate imprimitivity groupoid. Thus we have the
	following immediate corollary.
	
	\begin{cor}
	\label{cor:impxg}
		Let $X^G$ be the imprimitivity groupoid associated to the right $G$-space $X$. Then $X^G$ is naturally isomorphic to $L$.
	\end{cor}

In particular, note that Corollary \ref{cor:impxg} gives an identification of $X/G$ with $\lo$ via the map $\rho$.

\begin{cor}
	Under the homeomorphism $\rho : X/G \to \lo$, we can identify $\rho^*\B$ with $\B$. Also, under this identification and the 
	isomorphism $X^G \cong L$, the action of $X^G$ on $\rho^*\B$ agrees with that of $L$ on $\B$.
\end{cor}
\begin{proof}
	We have identified $X/G$ with $\lo$ by associating the orbit $[\gamma]$ to the unit $r(\gamma)$. It is then obvious that we can 
	identify the pullback bundle $\rho^*\B$ with $\B$, and thus $\rho^*B$ with $B$ as well.

	For any $\gamma, \eta \in X$ with $s(\gamma) = s(\eta)$, the isomorphism $\Phi : X^G \to L$ takes $[\gamma, \eta] \in L^G$ to 
	$\gamma \eta^{-1} \in L$. Recall that the action of $X^G$ on $\rho^*\B$ is given by a family $\sigma = \{\sigma_{[\gamma, \eta]}\}$, 
	where $\sigma_{[\gamma, \eta]} : \rho^*\B_{s([\gamma, \eta])} \to \rho^*\B_{r([\gamma, \eta])}$ is defined by
	\[
		\sigma_{[\gamma, \eta]}(a) = \alpha_{\gamma \eta^{-1}}(a).
	\]
	Here we have identified $\rho^*\B_{s([\gamma, \eta])}$ with $\B_{r(\eta)}$ and $\rho^*\B_{r([\gamma, \eta])}$ with $\B_{r(\gamma)}$. 
	But $\alpha_{\gamma \eta^{-1}}(a) = \alpha_{\Phi([\gamma, \eta])}(a)$, so the actions match up.
\end{proof}

We could replace $G$ by $H$ in the preceding situation, and everything would work equally well. That is, we could define the principal 
right $H$-space
\[
	Y = s^{-1}(\ho) = Z \sqcup H,
\]
and let $Y^H$ be the associated imprimitivity groupoid. We let $\psi : Y/H \to \lo$ denote the map
$\psi([\gamma]) = r(\gamma)$.
Then all of our previous results hold, and the proofs are nearly identical in this situation. To summarize:

\begin{prop}
	The space $Y$ is an $(L, H)$-equivalence, and consequently $Y^H$ is naturally isomorphic to $L$. Furthermore, we can identify 
	$\psi^*\B$ with $\B$, and the actions of $Y^H$ on $\psi^*\B$ and $L$ on $\B$ agree under this identification.
\end{prop}
	
Putting these results together establishes an equivalence between $(\B, L, \beta)$ and its restriction 
$(\A, G, \alpha)$ to $G$ (or alternatively, its restriction to $H$).
	
\begin{cor}
	Let $(\B, L, \beta)$ be a dynamical system, and let $(\A, G, \alpha)$ (respectively, $(\mathcal{C}, H, \tau)$) denote the restriction to $G$ (respectively, $H$). Then 
	$\B \rtimes_\beta L$ is Morita equivalent to $\A \rtimes_\alpha G$ (respectively, $\mathcal{C} \rtimes_\tau H$) via a completion of the pre-imprimitivity bimodule 
	$\Gamma_c(X, s^*\B)$ (respectively, $\Gamma_c(Y, s^*\B)$). Furthermore, this Morita equivalence descends to the level of reduced crossed products, 
	and $\Gamma_c(X, s^*\B)$ also completes to a $\B \rtimes_{\beta, r} L - \A \rtimes_{\alpha, r} G$-imprimitivity bimodule (respectively, 
	$\B \rtimes_{\beta, r} L - \mathcal{C} \rtimes_{\tau, r} H$-imprimitivity bimodule).
\end{cor}
	
In light of these results, Theorem \ref{thm:kwexact} simplifies to the following special case.

\begin{cor}
\label{cor:kwexactSpecialCase}
	Let $G$ and $H$ be locally compact Hausdorff groupoids with Haar systems, let $Z$ be a $(G, H)$-equivalence, and let $L$ denote 
	the associated linking groupoid. Suppose that $(\B, L, \beta)$ is a separable groupoid dynamical system and $J \subseteq B$ is an 
	$L$-invariant ideal.
	\begin{thmenum}
		\item If $(\A, G, \alpha)$ denotes the restriction of $(\B, L, \beta)$ to $G$ and $I = J \cap A$, then the sequence
		\[
			0 \to \mathcal{J} \rtimes_{\beta, r} L \to \B \rtimes_{\beta, r} L \to \B/\mathcal{J} \rtimes_{\beta, r} L \to 0
		\]
		is exact if and only if
		\[
			0 \to \I \rtimes_{\alpha, r} G \to \A \rtimes_{\alpha, r} G \to \A/\I \rtimes_{\alpha, r} G \to 0
		\]
		is exact.
		
		\item If $(\mathcal{C}, H, \tau)$ denotes the restriction of $(\B, L, \beta)$ to $H$ and $K = J \cap A$, then the sequence
		\[
			0 \to \mathcal{J} \rtimes_{\beta, r} L \to \B \rtimes_{\beta, r} L \to \B/\mathcal{J} \rtimes_{\beta, r} L \to 0
		\]
		is exact if and only if
		\[
			0 \to \mathcal{K} \rtimes_{\tau, r} H \to \mathcal{C} \rtimes_{\tau, r} H \to \mathcal{C}/\mathcal{K} \rtimes_{\tau, r} H \to 0
		\]
		is exact.
	\end{thmenum}
\end{cor}
	
The first half of this corollary tells us in particular that if $G$ is exact, then the linking groupoid $L$ is exact. We'd like to also use the 
second half of the corollary to deduce that the exactness of $L$ descends to $H$. However, to show that $H$ is exact, we must be able to 
take \emph{any} dynamical system $(\A, H, \alpha)$ and \emph{any} invariant ideal $I \subseteq A$, and show that
\[
	0 \to \I \rtimes_{\alpha, r} H \to \A \rtimes_{\alpha, r} H \to \A/\I \rtimes_{\alpha, r} H \to 0
\]
is exact. Therefore, to use Corollary \ref{cor:kwexactSpecialCase} we need to somehow create a new dynamical system $(\B, L, \beta)$ 
that restricts to $(\A, H, \alpha)$. This problem really amounts to the following: given two groupoids $G$ and $H$, a 
$(G,H)$-equivalence $Z$, and a dynamical system $(\A, H, \alpha)$, it is possible to ``induce'' an upper semicontinuous bundle 
$\A^Z \to \go$ and an action $\alpha^Z$ of $G$ on $\A^Z$? A trick for doing so originated in \cite{KMRW} for continuous
$C^*$-bundles, and it was then extended to upper semicontinuous bundles in \cite{BG} and \cite{jonbrown}. The details are worked out 
thoroughly in Section 6.4 of \cite{jonbrown}, 
so we summarize the general construction here.

Let $G$ and $H$ be groupoids, and suppose $Z$ is a $(G, H)$-equivalence and $(\A, H, \alpha)$ is a dynamical system. Consider the 
pullback 
\[
	s_Z^*\A = \bigl\{(z, a) \in Z \times \A : s(z) = p(a) \bigr\}.
\]
%
 	Then the bundle $s_Z^*\A$ is a (not necessarily locally compact) right $H$-space with respect to action defined by
	\begin{equation}
	\label{eq:EquivBundleAction}
		(z, a) \cdot \eta = \bigl(z \cdot \eta, \alpha_\eta^{-1}(a) \bigr).
	\end{equation}
	for $(z, a) \in s_Z^*\A$ and $\eta \in H$ \cite[Prop. 6.33]{jonbrown}.
%
%
We can then define $\A^Z$ to be the quotient space $s^*\A/H$. There is a continuous surjection $p^Z : \A^Z \to \go$ 
given by
\[
	p^Z([z, a]) = r_Z(z).
\]
Note that $p^Z$ is well-defined: if $[z, a]$ and $\eta \in H$, then 
\[
	p^Z([z \cdot \eta, \alpha_\eta^{-1}(a)]) = r(z \cdot \eta) = r(z).
\]
Furthermore, this map is open and makes $\A^Z$ into an upper semicontinuous $C^*$-bundle over $\go$. This is proven in 
Proposition 2.15 of \cite{KMRW} for continuous bundles, and the upper semicontinuous case is handled in Proposition 6.33 of \cite{jonbrown}.

Finally, we can define a continuous action of $G$ on the left of $\A^Z$ as follows: for $[z, a] \in \A^Z$ and $\eta \in G$, define
\[
	\alpha_\eta^Z ([z, a]) = [\eta \cdot z, a].
\]
%
	Then the family $\alpha^Z = \{\alpha^Z_\eta\}_{\eta \in G}$ defines a continuous action of $G$ on $\A^Z$, so $(\A^Z, G, \alpha^Z)$ is a 
	groupoid dynamical system.

\begin{defn}
	The dynamical system $(\A^Z, G, \alpha^Z)$ is called the \emph{induction} \index{induction of a bundle} of $(\A, H, \alpha)$ to $G$.
\end{defn}
	
Now let $G$ and $H$ be groupoids, $Z$ a $(G, H)$-equivalence, and $L$ the linking groupoid. We can now use the induction to construct
a dynamical system $(\B, L, \beta)$ extending a given system $(\A, H, \alpha)$.
	
\begin{prop}
	Define a bundle $q : \B \to \lo$ by
	\[
		\B = \A^Z \sqcup \A, \quad q = p^Z \sqcup p.
	\]
	Then $\B$ is an upper semicontinuous $C^*$-bundle, and $L$ acts on $\B$ via $*$-isomorphisms as follows:
	\begin{thmenum}
		\item if $\eta \in G$, then $\beta_\eta : \B_{s(\eta)} \to \B_{r(\eta)}$ is given by
		\[
			\alpha_\eta^Z : \A_{s(\eta)}^Z \to \A_{r(\eta)}^Z;
		\]
		
		\item if $\gamma \in H$, $\B_\gamma : \B_{s(\gamma)} \to \B_{r(\gamma)}$ is $\alpha_\gamma : \A_{s(\gamma)} \to \A_{r(\gamma)}$;
		\item if $z \in Z$, then $\B_{s(z)} = \A_{s(z)}$ and $\B_{r(z)} = \A_{r(z)}^Z$, and we put
			\[
				\beta_z(a) = [z, a];
			\]
		\item if $\bar{z} \in \Zop$, then $\B_{s(\bar{z})} = \A_{r(\bar{z})}^Z$ and $\B_{r(\bar{z})} = \A_{s(z)}$, and 
			\[
				\beta_{\bar{z}}([z, a]) = \beta_z^{-1}([z, a]) = a.
			\]
	\end{thmenum}
	This makes $(\B, L, \beta)$ into a dynamical system. Furthermore, $\B\vert_\ho = \A$ and $\beta_\gamma = \alpha_\gamma$ for all 
	$\gamma \in H$, so this dynamical system restricts to $(\A, H, \alpha)$ on $H$.
\end{prop}
\begin{proof}
	It is clear that $q : \B \to \lo$ defines an upper semicontinuous $C^*$-bundle. It is also immediate that $\beta_k$ is a $*$-isomorphism
	from $\B_{s(k)}$ to $\B_{r(k)}$ for all $k \in L$. Given $(k, l) \in \ltwo$, we certainly have $\beta_{kl} = \beta_k \circ \beta_l$ if $k, l \in G$
	or $k, l \in H$. Suppose that $\gamma \in G$ and $z \in Z$ with $s(\gamma) = r(z)$. Then
	\[
		\beta_{\gamma \cdot z}(a) = [\gamma \cdot z, a] = \alpha_\gamma^Z([z, a]) = \beta_\gamma([z, a]) = \beta_\gamma \bigl( \beta_z(a) \bigr).
	\]
	Similar computations work in the cases $\gamma \in G$ and $\bar{z} \in \Zop$, $\eta \in H$ and $z \in Z$, and $\eta \in H$ and $\bar{z} \in
	\Zop$. Suppose then that $z \in Z$ and $\bar{w} \in \Zop$ with $r(\bar{w}) = s(z)$. Then
	\begin{align*}
		\beta_{z \bar{w}}([w \cdot \gamma, a]) &= \alpha^Z_{[z, w]}([w \cdot \gamma, a]) \\
			&= [z \cdot \gamma, a] \\
			&= \bigl[ z, \alpha_\gamma(a) \bigr] \\
			&= \beta_z \bigl( \alpha_\gamma(a) \bigr) \\
			&= \beta_z \circ \beta_{\bar{w}} \bigl( [w, \alpha_\gamma(a)] \bigr) \\
			&= \beta_z \circ \beta_{\bar{w}} ([w \cdot \gamma, a]).
	\end{align*}
	Thus $\beta$ defines an action. To see that it is continuous, it suffices to work with $G$, $H$, $Z$, and $\Zop$ separately. We already
	know that the restrictions of $\beta$ to $G$ and $H$ are continuous, so suppose $z_i \to z$ in $Z$ and $a_i \to a$ in $\A$. Then
	\[
		\beta_{z_i}(a_i) = [z_i, a_i] \to [z, a].
	\]
	On the other hand, suppose $\bar{z}_i \to \bar{z}$ in $\Zop$ and $[z_i, a_i] \to [z, a]$ in $\A^Z$. Pass to a subnet. Then we can find
	$\eta_i \in H$ such that $(z \cdot \eta_i, \alpha_{\eta_i}^{-1}(a_i)) \to (z, a)$. Now $z_i \to z$ and $z_i \cdot \eta_i \to z$, so the properness
	of the $H$-action on $Z$ guarantees that $\eta_i$ has a convergent subnet. Moreover, this subnet must converge to $s(z)$. Pass to this
	subnet, relabel, and observe that
	\[
		(z_i, a_i) = \bigl( z \cdot \eta_i, \alpha_{\eta_i}^{-1}(a_i) \bigr) \cdot \eta_i^{-1} \to (z, a) \cdot s(z) = (z, a).
	\]
	Thus every subnet of $\{a_i\}$ has a subnet converging to $a$, so $a_i \to a$. Therefore,
	\[
		\beta_{\bar{z}_i}([z_i, a_i]) = \beta_{z_i}^{-1}([z_i, a_i]) = a_i \to a = \beta_{\bar{z}}([z, a]),
	\]
	so the restriction of $\beta$ to $\Zop$ is continuous as well.
	Finally, $(\B, L, \beta)$ restricts to $(\A, H, \alpha)$ simply by construction.
\end{proof}
	
	\begin{thm}
	\label{thm:ExactEquiv}
		Let $G$ and $H$ be equivalent groupoids. If $G$ is exact, then so is $H$.
	\end{thm}
	\begin{proof}
		Let $L$ be the associated linking groupoid. We have already observed that if $G$ is exact, then $L$ is exact. Let $(\A, H, \alpha)$ 
		be a dynamical system, $I$ an $H$-invariant ideal of $A$, and let $(\B, L, \beta)$ and $(\mathcal{J}, L, \beta)$ be the inductions of 
		these systems to $L$, as described above. Then
		\[
			0 \to \mathcal{J} \rtimes_{\beta, r} L \to \B \rtimes_{\beta, r} L \to \B/\mathcal{J} \rtimes_{\beta, r} \to 0
		\]
		is exact, since $L$ is exact. But then Corollary \ref{cor:kwexactSpecialCase} implies that the sequence
		\[
			0 \to \I \rtimes_{\alpha, r} H \to \A \rtimes_{\alpha, r} H \to \A/\I \rtimes_{\alpha, r} \to 0
		\]
		is exact. Therefore, $H$ is exact.
	\end{proof}
	
\section{Application: Transitive Groupoids}
In this section we present a brief application of Theorem \ref{thm:ExactEquiv}. First recall the following well-known result
for discrete groups.

\begin{thm}[{\cite[Thm. 5.2]{kw99}}]
\label{thm:DiscreteExact}
	Let $G$ be a discrete group. Then $G$ is exact if and only if $C_r^*(G)$ is an exact $C^*$-algebra.
\end{thm}

It seems plausible that such a theorem should hold more generally for \'{e}tale groupoids. Indeed, we can use Theorem \ref{thm:ExactEquiv} to prove
a particularly cute special case. Recall that a groupoid $G$ is \emph{transitive} if given $u, v \in \go$, there is a $\gamma \in G$ such that $s(\gamma) = u$ 
and $r(\gamma)=v$. Provided $G$ is second countable, \cite[Ex. 2.2]{MRW} guarantees that $G$ is equivalent to any of one of its isotropy groups. 

Let $G$ be a second countable, transitive groupoid \index{transitive groupoid} with discrete isotropy. (Note that this includes all transitive \'{e}tale groupoids.) 
Fix $u \in \go$, and let $S_u$ be the isotropy group at $u$. Assume that $C_r^*(G)$ is exact. Since $C_r^*(S_u)$ is Morita equivalent to $C_r^*(G)$,
$C_r^*(S_u)$ is also exact. Theorem \ref{thm:DiscreteExact} then implies that $S_u$ is an exact group. But Theorem \ref{thm:ExactEquiv} shows that exactness is preserved under groupoid equivalence, so $G$ is also exact. Since $C_r^*(G)$ is exact whenever $G$ is exact by \cite[Thm. 6.14]{lalonde2014}, we have proven:

\begin{thm}
\label{thm:etale}
	Let $G$ be a transitive groupoid with discrete isotropy. Then $G$ is exact if and only if $C_r^*(G)$ is exact.
\end{thm}



\end{document}